\documentclass{article}      

\title{Several properties of hypergeometric Bernoulli numbers}  
\author{
Miho Aoki
\\
\small Department of Mathematics\\
\small Shimane University\\ 
\small Matsue, Shimane, Japan\\ 
\small \texttt{aoki@riko.shimane-u.ac.jp}\\\\  
Takao Komatsu\thanks{Corresponding author.
}\\ 
\small School of Mathematics and Statistics\\
\small Wuhan University\\
\small Wuhan 430072 China\\
\small \texttt{komatsu@whu.edu.cn}\\\\
Gopal Krishna Panda\\
\small National Institute of Technology\\ 
\small Rourkela, India\\  
\small \texttt{gkpanda\_nit@rediffmail.com} 
}

\date{
\small MR Subject Classifications:  11A55, 11B68, 11B37, 11C20, 15A15, 33C15, 05A15, 05A19. 
}

\usepackage{amsmath,amsthm}
\usepackage{amssymb}
\usepackage{setspace}

\def\fl#1{\left\lfloor#1\right\rfloor}

\def\stf#1#2{\left[#1\atop#2\right]}

\newtheorem{theorem}{Theorem}
\newtheorem{Prop}{Proposition}
\newtheorem{Cor}{Corollary}
\newtheorem{Lem}{Lemma}
\newtheorem{ex}{Example}

\allowdisplaybreaks[4]

\begin{document}             

\maketitle                   

\begin{abstract}  
In this paper, we give the determinant expressions of the hypergeometric Bernoulli numbers,  and  some relations between the hypergeometric  and the classical Bernoulli numbers which include Kummer's congruences. By applying Trudi's formula, we have some different expressions and inversion relations.   
We also determine explicit forms of convergents of the generating function of the hypergeometric Bernoulli numbers, from which several identities for hypergeometric Bernoulli numbers are given.  

\end{abstract}

\section{Introduction} 

Denote ${}_1 F_1(a;b;z)$ be the confluent hypergeometric function defined by 
$$
{}_1 F_1(a;b;z)=\sum_{n=0}^\infty\frac{(a)^{(n)}}{(b)^{(n)}}\frac{z^n}{n!}  
$$ 
with the rising factorial $(x)^{(n)}=x(x+1)\dots(x+n-1)$ ($n\ge 1$) and $(x)^{(0)}=1$. 
For $N\ge 1$, define hypergeometric Bernoulli numbers $B_{N,n}$ (\cite{HN1,HN2,Ho1,Ho2,Kamano2}) by 
\begin{equation}  
\frac{1}{{}_1 F_1(1;N+1;x)}=\frac{x^N/N!}{e^x-\sum_{n=0}^{N-1}x^n/n!}=\sum_{n=0}^\infty B_{N,n}\frac{x^n}{n!}\,.   
\label{def:hgb} 
\end{equation}   
When $N=1$, $B_{1,n}=B_n$ are classical Bernoulli numbers, defined by 
$$
\frac{x}{e^x-1}=\sum_{n=0}^\infty B_n\frac{x^n}{n!}\,. 
$$  
In addition, define hypergeometric Bernoulli polynomials $B_{N,n}(z)$ (\cite{HK}) by the generating function 
$$
\frac{e^{x z}}{{}_1 F_1(1;N+1;x)}=\sum_{n=0}^\infty B_{N,n}(z)\frac{x^n}{n!}\,. 
$$ 
It is known (\cite{Ng}) that 
\begin{equation} 
\sum_{m=0}^n\binom{n}{m}B_{1,m}(x)B_{2,n-m}(y)=B_{2,n}(x+y)-\frac{n}{2}B_{1,n-1}(x+y)\quad(n\ge 1)\,. 
\label{cbf}
\end{equation}

Many kinds of generalizations of the Bernoulli numbers have been considered by many authors. For example, Poly-Bernoulli number, multiple Bernoulli numbers, Apostol Bernoulli numbers, multi-poly-Bernoulli numbers, degenerated Bernoulli numbers, various types of $q$-Bernoulli numbers, Bernoulli Carlitz numbers. One of the advantages of hypergeometric numbers is the natural extension of determinant expressions of the numbers. 

In \cite{KY}, some determinant expressions of hypergeometric Cauchy numbers are considered.  
In this paper, we shall give the similar determinant expression of hypergeometric Bernoulli numbers and their generalizations.  
Then we study some relations between the hypergeometric Bernoulli numbers and the classical Bernoulli numbers which include Kummer's congruences. Furthermore,  by applying Trudi's formula, we also have some different expressions and inversion relations.   
We also determine explicit forms of convergents of the generating function of the hypergeometric Bernoulli numbers, from which several identities for hypergeometric Bernoulli numbers are given.

\section{Some basic properties of hypergeometric Bernoulli numbers}   

From the definition (\ref{def:hgb}), we have 
\begin{align*}  
\frac{x^N}{N!}&=\left(\sum_{i=0}^\infty\frac{x^{i+N}}{(i+N)!}\right)\left(\sum_{m=0}^\infty B_{N,m}\frac{x^m}{m!}\right)\\
&=x^N\sum_{n=0}^\infty\sum_{m=0}^n\frac{x^{n-m}}{(n-m+N)!}B_{N,m}\frac{x^m}{m!}\\
&=\sum_{n=0}^\infty\sum_{m=0}^n\frac{B_{N,m}}{(n-m+N)!m!}x^{N+n}\,. 
\end{align*} 
Hence, for $n\ge 1$, we have the following.  

\begin{Prop}\label{prop:relation}  
$$ 
\sum_{m=0}^n\binom{N+n}{m}B_{N,m}=0\,.  
$$ 
\label{prp0} 
\end{Prop} 

\noindent  
{\it Remark.}  
When $N=1$, we have a famous identity for Bernoulli numbers.  
$$ 
\sum_{m=0}^n\binom{n+1}{m}B_{m}=0\quad(n\ge 1)\,.  
$$ 
If Bernoulli numbers $\mathfrak B_n$ are defined by 
$$
\frac{x}{1-e^{-x}}=\sum_{n=0}^\infty\mathfrak{B}_n\frac{x^n}{n!}\,,
$$ 
then it holds that 
$$
\sum_{m=0}^n\binom{n+1}{m}\mathfrak{B}_{m}=n+1\quad(n\ge 1)\,.
$$ 
Notice that $B_n=(-1)^n\mathfrak{B}_n$ ($n\ge 0$). 
\medskip 

By using Proposition \ref{prp0} or 
\begin{equation}  
B_{N,n}=-\sum_{k=0}^{n-1}\dfrac{\binom{N+n}{k}}{\binom{N+n}{n}}B_{N,k} 
\label{eq0}
\end{equation}  
with $B_{N,0}=1$ ($N\ge 1$), 
some values of $B_{N,n}$ ($0\le n\le 9$) are explicitly given by the following. 
{\small
\begin{align*}  
B_{N,0}=&1\,,\\ 
B_{N,1}=&-\frac{1}{N+1}\,,\\
B_{N,2}=&\frac{2}{(N+1)^2(N+2)}\,,\\ 
B_{N,3}=&\frac{3!(N-1)}{(N+1)^3(N+2)(N+3)}\,,\\ 
B_{N,4}=&\frac{4!(N^3-N^2-6 N+2)}{(N+1)^4(N+2)^2(N+3)(N+4)}\,,\\
B_{N,5}=&\frac{5!(N-1)(N^3-3 N^2-14 N+2)}{(N+1)^5(N+2)^2(N+3)(N+4)(N+5)}\,,\\
B_{N,6}=&\frac{6!(N^7-3N^6-49N^5-57N^4+222N^3+264N^2-198N+12)}{(N+5)(N+4)(N+6)(N+3)^2(N+2)^3(N+1)^6}\,,\\
B_{N,7}=&\frac{7!(N-1))(N^7-7N^6-81N^5-37N^4+766N^3+1048N^2-390N+12)}{(N+6)(N+5)(N+4)(N+3)^2(N+2)^3(N+7)(N+1)^7} \,,\\
B_{N,8}=&\frac{8!}{(N+7)
(N+6)(N+5)(N+3)^2(N+8)(N+4)^2(N+2)^4(N+1)^8}\,\\
& \times (N^{11}-8N^{10}-172N^9-354N^8+3265N^7+13498N^6+1164N^5\,\\
& -46836N^4-23650N^3+38356N^2-6096N+96)\,,\\
B_{N,9}=&\frac{9!(N-1)}{
(N+8)(N+7)(N+6)(N+5)(N+4)^2(N+2)^4(N+9)(N+3)^3(N+1)^9}\,\\
&\times (N^{12}-11N^{11}-284N^{10}-846N^9+8559N^8+59067N^7+79142N^6\,\\
&-257992N^5-768982N^4-346890N^3+342588N^2-33936N+288)\,.\\
\end{align*} 
}
In general, we have an explicit expression of $B_{N,n}$.  

\begin{Prop} 
For $N,n\ge 1$, we have 
$$
B_{N,n}=n!\sum_{k=1}^n\sum_{i_1+\cdots+i_k=n\atop i_1,\dots,i_k\ge 1}\frac{(-N!)^k}{(N+i_1)!\cdots(N+i_k)!}\,.  
$$ 
\label{prp1}  
\end{Prop}   

\noindent 
{\it Remark.}  
In the later section about Trudi's formula, we see a different expression of $B_{N,n}$ in Corollary \ref{cor1}.   Further, an inversion expression can be obtained:  
$$
\binom{N+n}{N}^{-1}=\sum_{k=1}^n(-1)^k\sum_{i_1+\cdots+i_k=n\atop i_1,\dots,i_k\ge 1}\binom{n}{i_1,\dots,i_k}B_{N,i_1}\cdots B_{N,i_k}\,,  
$$ 
where $\binom{n}{t_1, \dots,t_k}=\frac{n!}{t_1 !\cdots t_k !}$ are the multinomial coefficients.

\begin{proof}[Proof of Proposition \ref{prp1}]
The proof can be done by induction on $n$.  Here, we shall prove directly by using the generating function.  From the definition (\ref{def:hgb}), we have 
\begin{align*}  
\sum_{n=0}^\infty B_{N,n}\frac{x^n}{n!}&=\frac{x^N/N!}{e^x-\sum_{i=0}^{N-1}x^i/i!}\\
&=\frac{1}{\frac{N!}{x^N}(e^x-\sum_{i=0}^N x^i/i!)+1}\\
&=\sum_{k=0}^\infty\left(-\frac{N!}{x^N}\left(e^x-\sum_{i=0}^N\frac{x^i}{i!}\right)\right)^k\,.   
\end{align*}
The proposition immediately follows by comparing coefficients of both sides. 
\end{proof}

We also have a different expression of $B_{N,n}$ with binomial coefficients. The proof is similar to that of Proposition \ref{prp1} and omitted.  
 
\begin{Prop} 
For $N,n\ge 1$, we have 
$$
B_{N,n}=n!\sum_{k=1}^n\binom{n+1}{k+1}\sum_{i_1+\cdots+i_k=n\atop i_1,\dots,i_k\ge 0}\frac{(-N!)^k}{(N+i_1)!\cdots(N+i_k)!}\,.  
$$ 
\label{prp2}  
\end{Prop} 

\section{Analog of Kummer's congruence}

Let $p$ be a prime number, and $\nu \geq 0$ be an integer.
If $m$ and $n$ are positive even integers with $m\equiv n \pmod{(p-1)p^{\nu}}$
and $m,n \not\equiv 0 \pmod{p-1}$, then we have
\begin{equation}\label{Kummer}
(1-p^{m-1})\frac{B_m}{m} \equiv (1-p^{n-1})\frac{B_n}{n} \pmod{p^{\nu+1}},
\end{equation}
and this is called Kummer's congruence (\cite[Corollary~5.14]{Washington}).
We get the similar congruence for the hypergeometric Bernoulli numbers $B_{N,n}$
if $N$ is $p$-adically close enough to $1$, that is, ${\rm ord}_p(N-1)$ is
enough large.
\begin{Lem}\label{lem:Kummer}
Let $p$ be a prime number. For $N\geq 1 $ and $n\geq 0$, we have 
$$
\prod_{k=0}^n \frac{(N+k)!}{N!} B_{N,n} \equiv \prod_{k=0}^n (1+k)! \ B_n \pmod{p^t},
$$
where $t={\rm ord}_p(N-1)$.
\end{Lem}
\begin{proof}
In the case $n=0$, the assertion is trivial.
Assume that the result is true up to $n-1$.
By Proposition~\ref{prop:relation}, we have
\begin{align*}
\prod_{k=0}^n \frac{(N+k)!}{N!} B_{N,n}  &=n! \prod_{k=0}^{n-1} \frac{(N+k)!}{N!} \times 
\binom{N+n}{n}B_{N,n} \\
&= n! \prod_{k=0}^{n-1} \frac{(N+k)!}{N!} \times  \left\{ -\sum_{m=0}^{n-1} \binom{N+n}{m} B_{N,m} 
\right\} \\
&=-\sum_{m=0}^{n-1} n! \binom{N+n}{m}
\prod_{k=m+1}^{n-1} \frac{(N+k)!}{N!} \times \prod_{k=0}^m \frac{(N+k)!}{N!}  B_{N,m} \\
&\equiv -\sum_{m=0}^{n-1} n! \binom{N+n}{m} \prod_{k=m+1}^{n-1} \frac{(N+k)!}{N!} \times 
\prod_{k=0}^m (1+k)! \ B_ m \pmod{p^t} \\
&\equiv -\prod_{k=0}^{n-1} (1+k)! \times n! \sum_{m=0}^{n-1} \binom{n+1}{m}  B_m \pmod{p^t}\\
&\equiv \prod_{k=0}^{n-1} (1+k)! \times n! \times \binom{n+1}{n} B_n \pmod{p^t} \\
&=\prod_{k=0}^n (1+k)! B_n \pmod{p^t}.
\end{align*}
\end{proof}
From this lemma, we have the following corollary.
\begin{Cor}
Let $p$ be a prime number, and $N,n\geq 1 , \nu \geq 0$ be integers with $n\not\equiv 0\pmod{p-1}$.
If ${\rm ord}_p(N-1)\geq \nu+1+{\rm ord}_p \left( \prod_{k=0}^n (1+k)! \right) +{\rm ord}_p(n)$,
then we have
\begin{equation*}
\frac{B_{N,n}}{n} \equiv \frac{B_n}{n } \pmod{p^{\nu+1}}.
\end{equation*}
\end{Cor}
Furthermore, by using (\ref{Kummer}), we have the following Proposition.
\begin{Prop}
Let $p$ be a prime number, and $\nu \geq 0$ be an
integer. If $m$ and $n$ are positive even integers with $m\geq n,\
m\equiv n \pmod{(p-1)p^{\nu}}$ and $m,n\not\equiv 0\pmod{(p-1)}$,
and ${\rm ord}_p(N-1)\geq \nu+1 +{\rm ord}_p\left( \prod_{k=0}^m (1+k)! \right)+
{\rm max} \left\{ {\rm ord}_p(m), {\rm ord}_p(n) \right\}$,
then we have
$$
(1-p^{m-1}) \frac{B_{N,m}}{m} \equiv (1-p^{n-1}) \frac{B_{N,n}}{n} \pmod{p^{\nu+1}}.
$$
\end{Prop}
\begin{ex}
Consider the case $p=5, m=6, n=2$.
For any integer $N$ satisfying
$$
{\rm ord}_5(N-1) \geq 1+{\rm ord}_5\left( \prod_{k=0}^6 (1+k)! \right) =4,
$$
we have
$$
\frac{B_{N,6}}{6} \equiv \frac{B_{N,2}}{2} \equiv 3 \pmod{5}.
$$
\end{ex}
\begin{ex}
Consider the case $p=5, m=22, n=2$.
For any integer $N$ satisfying
$$
{\rm ord}_5(N-1) \geq 2+{\rm ord}_5\left( \prod_{k=0}^{22} (1+k)! \right) =48,
$$
we have
$$
(1-5^{21})\frac{B_{N,22}}{22} \equiv  (1-5) \frac{B_{N,2}}{2} \equiv 8 \pmod{5}.
$$
\end{ex}
\section{Determinant expressions}

\begin{theorem}  
For $N,n\ge 1$, we have 
$$
B_{N,n}=(-1)^n n!\left|
\begin{array}{ccccc} 
\frac{N!}{(N+1)!}&1&&&\\  
\frac{N!}{(N+2)!}&\frac{N!}{(N+1)!}&&&\\ 
\vdots&\vdots&\ddots&1&\\ 
\frac{N!}{(N+n-1)!}&\frac{N!}{(N+n-2)!}&\cdots&\frac{N!}{(N+1)!}&1\\ 
\frac{N!}{(N+n)!}&\frac{N!}{(N+n-1)!}&\cdots&\frac{N!}{(N+2)!}&\frac{N!}{(N+1)!}
\end{array} 
\right|\,. 
$$ 
\label{th1} 
\end{theorem} 

\noindent 
{\it Remark.}   
When $N=1$, we have a determinant expression of Bernoulli numbers (\cite[p.53]{
Glaisher}): 
\begin{equation} 
B_n=(-1)^n n!\left|
\begin{array}{ccccc} 
\frac{1}{2!}&1&&&\\  
\frac{1}{3!}&\frac{1}{2!}&&&\\ 
\vdots&\vdots&\ddots&1&\\ 
\frac{1}{n!}&\frac{1}{(n-1)!}&\cdots&\frac{1}{2!}&1\\ 
\frac{1}{(n+1)!}&\frac{1}{n!}&\cdots&\frac{1}{3!}&\frac{1}{2!}
\end{array} 
\right|\,. 
\label{det:ber}
\end{equation}   

\begin{proof}[Proof of Theorem \ref{th1}]  
This Theorem is a special case of Theorem \ref{th1-r}.  
\end{proof}


\section{A relation between $B_{N,n}$ and $B_{N-1,n}$}

In this section, we show the following relation between $B_{N,n}$ and $B_{N-1,n}$.

\begin{Prop}
For $N\geq 2$ and $n\geq 1$, we have
$$
\begin{array}{rl}
B_{N,n} =& \displaystyle{ \frac{N}{N+n}  \left\{  \sum_{m=0}^{n-1} \
\sum_{1\leq i_m <\cdots <i_1 <i_0=n} B_{N-1,i_m}  \right. }
\\
\\
& \displaystyle{ \times \left.
\prod_{k=1}^m
B_{N-1,i_{k-1}-i_k+1} \binom{i_{k-1}}{i_{k-1}-i_k+1}
\frac{N}{N+i_k} \right\}.}
\end{array}
$$
\label{prp4}
\end{Prop}

\begin{ex}

\begin{itemize}
\item[{\rm (i)}] $\displaystyle{ B_{N,1}=\frac{N}{N+1} B_{N-1,1}}$
\item[{\rm (ii)}] $\displaystyle{ B_{N,2}=\frac{N}{N+2} \left\{ B_{N-1,2}+\frac{N}{N+1} B_{N-1,1} B_{N-1,2} \right\} }$
\item[{\rm (iii)}] \begin{align*}
B_{N,3} =&
\frac{N}{N+3} \left\{ B_{N-1,3}+\frac{N}{N+1} B_{N-1,1}B_{N-1,3}+\frac{3N}{N+2} B_{N-1,2}^2 \right.\\
&  \left. +\frac{3N^2}{(N+1)(N+2)}B_{N-1,1}B_{N-1,2}^2 \right\}
\end{align*}
\end{itemize}
\label{ex1}
\end{ex}

By using Proposition \ref{prp4} for $N=2$ and $B_n=0$ for odd $n\geq 3$, the numbers $B_{2,n} (0\leq n\leq 4)$ are explicitly given
by the classical Bernoulli numbers $B_n$ (cf. \cite[\S 9]{Ho1}).

{\small
\begin{align*}  
B_{2,0}&=B_1 (=1)\,,\\ 
B_{2,1}&=\frac{2}{3} B_1 \left(=-\frac{1}{3}\right)\,,\\
B_{2,2}&=\frac{1}{2} B_2 +\frac{1}{3} B_1 B_2 \left(=\frac{1}{2\times 3^2} \right)\,,\\ 
B_{2,3}&=\frac{3}{5} B_2^2+\frac{2}{5} B_1 B_2^2 \left(=\frac{1}{2\times 3^2\times 5}\right)\,,\\ 
B_{2,4}& =\frac{1}{3} B_4+\frac{2}{9} B_1 B_4+\frac{6}{5} B_2^3+\frac{4}{5} B_1 B_2^3 \left(=-\frac{1}{2\times 3^3\times 5}\right)\,. 
\end{align*} 
}

\begin{Lem}
For $N\geq 2$ and $n\geq 1$, we have
$$
B_{N,n}=\frac{N}{N+n} \left\{ B_{N-1,n}+\sum_{m=1}^{n-1} \binom{n}{n-m+1} B_{N,m}B_{N-1,n-m+1} \right\}.
$$
\label{lem1}
\end{Lem}

\begin{proof}
From the derivative of (\ref{def:hgb}), we have 
\begin{align*}
\sum_{k=0}^{\infty} \frac{B_{N-1,k}}{k!} x^k & =\left( \sum_{k=0}^{\infty} \frac{B_{N,k+1}}{k!} x^k \right) \left(1-\sum_{k=0}^{\infty}
\frac{B_{N-1,k}}{k!} x^k \right)  +\sum_{k=0}^{\infty} \frac{B_{N,k}}{k!} x^k.
\end{align*}
By $B_{N-1,0}=1$, we have
\begin{align*}
\sum_{k=0}^{\infty} \frac{B_{N-1,k}}{k!} x^k & =-\sum_{k=0}^{\infty} \sum_{\ell =0}^{\infty} B_{N,k+1}B_{N-1,\ell +1} 
\frac{x^{k+\ell+1}}{k!(\ell+1)!} +\sum_{k=0}^{\infty} \frac{B_{N,k}}{k!} x^k \\
& =-\sum_{n=1}^{\infty} \sum_{\ell =0}^{n-1} B_{N,n-\ell}B_{N-1,\ell +1} \binom{n}{\ell+1}
\frac{x^n}{n!} +\sum_{k=0}^{\infty} \frac{B_{N,k}}{k!} x^k \\
& =\sum_{n=1}^{\infty} \left\{ B_{N,n}- \sum_{\ell =0}^{n-1} \binom{n}{\ell+1}B_{N,n-\ell}B_{N-1,\ell +1} \right\}
\frac{x^n}{n!} +B_{N,0}.
\end{align*}
Therefore, we have
$$
B_{N-1,n}=B_{N,n}-\sum_{\ell=0}^{n-1} \binom{n}{\ell+1} B_{N,n-\ell}B_{N-1,\ell+1},
$$
for $n\geq 1$. By $B_{N-1,1}=-\frac{1}{N}$, we have 
$$
B_{N-1,n}=\frac{N+n}{N}B_{N,n}-\sum_{\ell=1}^{n-1} \binom{n}{\ell+1} B_{N,n-\ell}B_{N-1,\ell+1},
$$
and hence,
\begin{align*}
B_{N,n} & =\frac{N}{N+n} \left\{ B_{N-1,n}+\sum_{\ell=1}^{n-1} \binom{n}{\ell+1} B_{N,n-\ell} B_{N-1,\ell+1} \right\}\\
& =\frac{N}{N+n} \left\{ B_{N-1,n}+\sum_{m=1}^{n-1} \binom{n}{n-m+1} B_{N,m} B_{N-1,n-m+1} \right\} .
\end{align*}
\end{proof}

\begin{proof}[Proof of Proposition \ref{prp4}]
We give the proof by induction for $n$.
In the case $n=1$, the assertion means $B_{N,1}=\frac{N}{N+1} B_{N-1,1}$, and this equality follows from $B_{N,1}=-\frac{1}{N+1}$ and $B_{N-1,1}=-\frac{1}{N}$.
Assume that the assertion holds up to $n-1$. By Lemma \ref{lem1}, we have

{\small
\begin{align*}
B_{N,n} =&\frac{N}{N+n} \left\{ B_{N-1,n}+\sum_{i_1=1}^{n-1}\binom{n}{n-i_1+1} B_{N,i_1}B_{N-1,n-i_1+1}  \right\} \\
=& \frac{N}{N+n} \left\{ B_{N-1,n}+\sum_{i_1=1}^{n-1} \binom{n}{n-i_1+1} B_{N-1,n-i_1+1} \frac{N}{N+i_1} 
 \right.  \\
&  \left. \times \left(  \sum_{m=0}^{i_1-1} \sum_{1\leq i_{m+1} <\cdots <i_2<i_1} B_{N-1,i_{m+1}}\prod_{k=2}^{m+1} B_{N-1,i_{k-1}-i_k+1} \binom{i_{k-1}}{i_{k-1}-i_k+1} 
\frac{N}{N+i_k} \right) \right\}\\
=& \frac{N}{N+n} \left\{ B_{N-1,n}+\sum_{i_1=1}^{n-1} \sum_{m=0}^{i_1-1} \sum_{1\leq i_{m+1}< \cdots <i_2< i_1 } B_{N-1,i_{m+1} }\right. \\
&  \left.  \times \prod_{k=1}^{m+1} B_{N-1,i_{k-1}-i_k+1} 
\binom{i_{k-1}}{i_{k-1}-i_k+1} \frac{N}{N+i_k}  \right\}\\
=& \frac{N}{N+n} \left\{ B_{N-1,n} +\sum_{i_1=1}^{n-1}\sum_{\ell =1}^{i_1} \sum_{1\leq i_{\ell}< \cdots <i_2 < i_1} B_{N-1,i_{\ell } }\right. \\
& \left.  \times \prod_{k=1}^{\ell} B_{N-1,i_{k-1}-i_k+1} 
\binom{i_{k-1}}{i_{k-1}-i_k+1} \frac{N}{N+i_k} \Biggr) \right\}\\
=& \frac{N}{N+n} \left\{ B_{N-1,n} +  \sum_{m=1}^{n-1} 
\sum_{1\leq i_m <\cdots <i_1 \leq n-1} B_{N-1,i_m}  \right. \\
& \left.
\times \prod_{k=1}^m
B_{N-1,i_{k-1}-i_k+1} \binom{i_{k-1}}{i_{k-1}-i_k+1}
\frac{N}{N+i_k} \right\}\\
=& \frac{N}{N+n}  \left\{  \sum_{m=0}^{n-1} \
\sum_{1\leq i_m <\cdots <i_1 <i_0=n} B_{N-1,i_m}  \right. 
\\
& \times \left.
\prod_{k=1}^m
B_{N-1,i_{k-1}-i_k+1} \binom{i_{k-1}}{i_{k-1}-i_k+1}
\frac{N}{N+i_k} \right\}.
\end{align*}
}
\end{proof}

\section{Multiple hypergeometric Bernoulli numbers}  
 
For positive integers $N$ and $r$, define the {\it higher order hypergeometric Bernoulli numbers} $B_{N,n}^{(r)}$ (\cite{Kamano2,NC}) by the generating function 
\begin{equation} 
\frac{1}{{}_1F_1(1;N+1+x)^r}
=\left(\frac{x^N/N!}{e^x-\sum_{n=0}^{N-1}x^n/n!}\right)^r=\sum_{n=0}^\infty B_{N,n}^{(r)}\frac{x^n}{n!}\,. 
\label{def:higher-hgb} 
\end{equation} 
The higher order hypergeometric Bernoulli polynomials $B_{N,n}^{(r)}(x)$ are studied in \cite{HK}, so that $B_{N,n}^{(r)}=B_{N,n}^{(r)}(0)$.  

From the definition (\ref{def:higher-hgb}), we have 
\begin{align*}  
\left(\frac{x^N}{N!}\right)^r&=\left(\sum_{i=0}^\infty\frac{x^{i+N}}{(i+N)!}\right)^r\left(\sum_{m=0}^\infty B_{N,m}^{(r)}\frac{x^m}{m!}\right)\\
&=x^{r N}\left(\sum_{l=0}^\infty\sum_{i_1+\cdots+i_r=l\atop i_1,\dots,i_r\ge 0}\frac{l!}{(N+i_1)!\cdots(N+i_r)!}\frac{x^l}{l!}\right)\left(\sum_{m=0}^\infty B_{N,m}^{(r)}\frac{x^m}{m!}\right)\\ 
&=x^{r N}\sum_{n=0}^\infty\sum_{m=0}^n\sum_{i_1+\cdots+i_r=n-m\atop i_1,\dots,i_r\ge 0}\binom{n}{m}\frac{(n-m)!}{(N+i_1)!\cdots(N+i_r)!}B_{N,m}^{(r)}\frac{x^n}{n!}\,. 
\end{align*} 
Hence, as a generalization of Proposition (\ref{prp0}), for $n\ge 1$, we have the following.  

\begin{Prop}  
$$ 
\sum_{m=0}^n\sum_{i_1+\cdots+i_r=n-m\atop i_1,\dots,i_r\ge 0}\frac{B_{N,m}^{(r)}}{m!(N+i_1)!\cdots(N+i_r)!}=0\,.  
$$ 
\label{prp0h} 
\end{Prop}

By using Proposition \ref{prp0h} or 
\begin{equation}  
B_{N,n}^{(r)}=-n!(N!)^r\sum_{m=0}^{n-1}\sum_{i_1+\cdots+i_r=n-m\atop i_1,\dots,i_r\ge 0}\frac{B_{N,m}^{(r)}}{m!(N+i_1)!\cdots(N+i_r)!}
\label{eq0h}
\end{equation}  
with $B_{N,0}^{(r)}=1$ ($N\ge 1$), 
some values of $B_{N,n}^{(r)}$ ($0\le n\le 4$) are explicitly given by the following. 

{\small
\begin{align*}  
B_{N,0}^{(r)}&=1\,,\\ 
B_{N,1}^{(r)}&=-\frac{r}{N+1}\,,\\
B_{N,2}^{(r)}&=\frac{2 r}{(N+1)^2(N+2)}\left(-(N+1)+\frac{r+1}{2}(N+2)\right)\,,\\ 
B_{N,3}^{(r)}&=\frac{3!r}{(N+1)^3(N+2)(N+3)}\biggl(-(N+1)^2+(r+1)(N+1)(N+3)\\
&\left.\quad -\frac{(r+1)(r+2)}{6}(N+2)(N+3)\right)\,,\\ 
B_{N,4}^{(r)}&=\frac{4!r}{(N+1)^4(N+2)^2(N+3)(N+4)}\biggl(-(N+1)^3(N+2)\\
&\quad+(r+1)(N+1)^2(N+2)(N+4)+\frac{r+1}{2}(N+1)^2(N+3)(N+4)\\
&\quad-\frac{(r+1)(r+2)}{2}(N+1)(N+2)(N+3)(N+4)\\
&\left.\quad+\frac{(r+1)(r+2)(r+3)}{4!}(N+2)^2(N+3)(N+4)\right)\,. 
\end{align*} 
}

As a generalization of Proposition \ref{prp1}, we have an explicit expression of $B_{N,n}^{(r)}$.  

\begin{Prop} 
For $N,n\ge 1$, we have 
$$
B_{N,n}^{(r)}=n!\sum_{k=1}^n(-1)^k\sum_{e_1+\cdots+e_k=n\atop e_1,\dots,e_k\ge1}M_r(e_1)\cdots M_r(e_k)\,,
$$
where 
\begin{equation} 
M_r(e)=
\sum_{i_1+\cdots+i_r=e\atop i_1,\dots,i_r\ge 0}\frac{(N!)^r}{(N+i_1)!\cdots(N+i_r)!}\,.  
\label{mre} 
\end{equation}  
\label{prp1h}  
\end{Prop}  

We shall introduce the Hasse-Teichm\"uller derivative 
in order to prove Proposition \ref{prp1h} easily.    
Let $\mathbb{F}$ be a field of any characteristic, $\mathbb{F}[[z]]$ the ring of formal power series in one variable $z$, and $\mathbb{F}((z))$ the field of Laurent series in $z$. Let $n$ be a nonnegative integer. We define the Hasse-Teichm\"uller derivative $H^{(n)}$ of order $n$ by 
$$
H^{(n)}\left(\sum_{m=R}^{\infty} c_m z^m\right)
=\sum_{m=R}^{\infty} c_m \binom{m}{n}z^{m-n}
$$
for $\sum_{m=R}^{\infty} c_m z^m\in \mathbb{F}((z))$, 
where $R$ is an integer and $c_m\in\mathbb{F}$ for any $m\geq R$. Note that $\binom{m}{n}=0$ if $m<n$.  

The Hasse-Teichm\"uller derivatives satisfy the product rule \cite{Teich}, the quotient rule \cite{GN} and the chain rule \cite{Hasse}. 
One of the product rules can be described as follows.  
\begin{Lem}  
For $f_i\in\mathbb F[[z]]$ $(i=1,\dots,k)$ with $k\ge 2$ and for $n\ge 1$, we have 
$$
H^{(n)}(f_1\cdots f_k)=\sum_{i_1+\cdots+i_k=n\atop i_1,\dots,i_k\ge 0}H^{(i_1)}(f_1)\cdots H^{(i_k)}(f_k)\,. 
$$ 
\label{productrule2}
\end{Lem} 

The quotient rules can be described as follows.  
\begin{Lem}  
For $f\in\mathbb F[[z]]\backslash \{0\}$ and $n\ge 1$,  
we have 
\begin{align} 
H^{(n)}\left(\frac{1}{f}\right)&=\sum_{k=1}^n\frac{(-1)^k}{f^{k+1}}\sum_{i_1+\cdots+i_k=n\atop i_1,\dots,i_k\ge 1}H^{(i_1)}(f)\cdots H^{(i_k)}(f)
\label{quotientrule1}
\\ 
&=\sum_{k=1}^n\binom{n+1}{k+1}\frac{(-1)^k}{f^{k+1}}\sum_{i_1+\cdots+i_k=n\atop i_1,\dots,i_k\ge 0}H^{(i_1)}(f)\cdots H^{(i_k)}(f)\,.
\label{quotientrule2} 
\end{align}   
\label{quotientrules}
\end{Lem}

\begin{proof}[Proof of Proposition \ref{prp1h}]  
Put $h(x)=\bigl(f(x)\bigr)^r$, where 
$$
f(x)=\dfrac{\sum_{i=N}^\infty\frac{x^i}{i!}}{\frac{x^N}{N!}}=\sum_{j=0}^\infty\frac{N!}{(N+j)!}x^j\,. 
$$ 
Since 
\begin{align*} 
\left.H^{(i)}(f)\right|_{x=0}&=\left.\sum_{j=i}^\infty\frac{N!}{(N+j)!}\binom{j}{i}x^{j-i}\right|_{x=0}\\
&=
\frac{N!}{(N+i)!}
\end{align*}  
by the product rule of the Hasse-Teichm\"uller derivative in Lemma \ref{productrule2}, we get 
\begin{align*} 
\left.H^{(e)}(h)\right|_{x=0}&=\sum_{i_1+\cdots+i_r=e\atop i_1,\dots,i_r\ge 0}\left.H^{(i_1)}(f)\right|_{x=0}\cdots\left.H^{(i_r)}(f)\right|_{x=0}\\
&=\sum_{i_1+\cdots+i_r=e\atop i_1,\dots,i_r\ge 0}\frac{N!}{(N+i_1)!}\cdots\frac{N!}{(N+i_r)!}=M_r(e)\,. 
\end{align*} 
Hence, by the quotient rule of the Hasse-Teichm\"uller derivative in Lemma \ref{quotientrules} (\ref{quotientrule1}), we have 
\begin{align*} 
\frac{B_{N,n}^{(r)}}{n!}&=\sum_{k=1}^n\left.\frac{(-1)^k}{h^{k+1}}\right|_{x=0}\sum_{e_1+\cdots+e_k=n\atop e_1,\dots,e_k\ge 1}\left.H^{(e_1)}(h)\right|_{x=0}\cdots\left.H^{(e_k)}(h)\right|_{x=0}\\
&=\sum_{k=1}^n(-1)^k\sum_{e_1+\cdots+e_k=n\atop e_1,\dots,e_k\ge1}M_r(e_1)\cdots M_r(e_k)\,. 
\end{align*} 
\end{proof}

Now, we can also show a determinant expression of $B_{N,n}^{(r)}$.   

\begin{theorem}  
For $N,n\ge 1$, we have 
$$
B_{N,n}^{(r)}=(-1)^n n!\left|
\begin{array}{ccccc} 
M_r(1)&1&&&\\  
M_r(2)&M_r(1)&&&\\ 
\vdots&\vdots&\ddots&1&\\ 
M_r(n-1)&M_r(n-2)&\cdots&M_r(1)&1\\ 
M_r(n)&M_r(n-1)&\cdots&M_r(2)&M_r(1) 
\end{array} 
\right|\,. 
$$ 
where $M_r(e)$ are given in $(\ref{mre})$. 
\label{th1-r}  
\end{theorem}  

\noindent 
{\it Remark.}  
When $r=1$ in Theorem \ref{th1-r}, we have the result in Theorem \ref{th1}.  

\begin{proof} 
For simplicity, put $A_{N,n}^{(r)}=(-1)^n B_{N,n}^{(r)}/n!$. Then, we shall prove that for any $n\ge 1$ 
\begin{equation}  
A_{N,n}^{(r)}=\left|
\begin{array}{ccccc} 
M_r(1)&1&&&\\  
M_r(2)&M_r(1)&&&\\ 
\vdots&\vdots&\ddots&1&\\ 
M_r(n-1)&M_r(n-2)&\cdots&M_r(1)&1\\ 
M_r(n)&M_r(n-1)&\cdots&M_r(2)&M_r(1) 
\end{array} 
\right|\,.
\label{aNnr}
\end{equation}   
When $n=1$, (\ref{aNnr}) is valid because 
$$
M_r(1)=\frac{r(N!)^r}{(N!)^{r-1}(N+1)!}=\frac{r}{N+1}=A_{N,1}^{(r)}\,. 
$$ 
Assume that (\ref{aNnr}) is valid up to $n-1$. Notice that 
by (\ref{eq0h}), we have 
$$
A_{N,n}^{(r)}=\sum_{l=1}^n(-1)^{l-1}A_{N,n-l}^{(r)}M_r(l)\,. 
$$ 
Thus, by expanding the first row of the right-hand side (\ref{aNnr}), it is equal to 
\begin{align*} 
&M_r(1)A_{N,n-1}^{(r)}-\left|
\begin{array}{ccccc} 
M_r(2)&1&&&\\  
M_r(3)&M_r(1)&&&\\ 
\vdots&\vdots&\ddots&1&\\ 
M_r(n-1)&M_r(n-3)&\cdots&M_r(1)&1\\ 
M_r(n)&M_r(n-2)&\cdots&M_r(2)&M_r(1) 
\end{array} 
\right|\\
&=M_r(1)A_{N,n-1}^{(r)}-M_r(2)A_{N,n-2}^{(r)}\\
&\qquad +\left|
\begin{array}{ccccc} 
M_r(3)&1&&&\\  
M_r(4)&M_r(1)&&&\\ 
\vdots&\vdots&\ddots&1&\\ 
M_r(n-1)&M_r(n-4)&\cdots&M_r(1)&1\\ 
M_r(n)&M_r(n-3)&\cdots&M_r(2)&M_r(1) 
\end{array} 
\right|\\
&=M_r(1)A_{N,n-1}^{(r)}-M_r(2)A_{N,n-2}^{(r)}+\cdots
+(-1)^{n-2}\left|
\begin{array}{cc}
M_r(n-1)&1\\
M_r(n)&M_r(1)
\end{array} 
\right|\\
&=\sum_{l=1}^n(-1)^{l-1}M_r(l)A_{N,n-l}^{(r)}=A_{N,n}^{(r)}\,.
\end{align*} 
Note that $A_{N,1}^{(r)}=M_r(1)$ and $A_{N,0}^{(r)}=1$. 
\end{proof}

\section{A relation between $B_{N,n}^{(r)}$ and $B_{N,n}$}
In this section, we show the following relation between $B_{N,n}^{(r)}$ and $B_{N,n}$.


\begin{Lem}
For $r,N\geq 1$ and $n\geq 0$, we have
$$
B_{N,n}^{(r)}=\sum_{n_1,\ldots,n_r \geq 0 \atop
n_1+\cdots +n_r=n} \frac{n!}{n_1! \cdots n_{r!}} B_{N,n_1} \cdots B_{N,n_r}.
$$
\end{Lem}
\begin{proof}
From the definition (\ref{def:higher-hgb}), we have
\begin{align*}
\sum_{n=0}^{\infty} B_{N,n}^{(r)} \frac{x^n}{n!} &=\left(
\sum_{n=0}^{\infty} B_{N,n} \frac{x^n}{n!} \right)^r\\
&=\sum_{n=0}^{\infty} \sum_{n_1,\ldots,n_r \geq 0 \atop
n_1+\cdots +n_r=n}  \frac{n!}{n_1! \cdots n_r!} B_{N,n_1} \cdots B_{N,n_r} \frac{x^n}{n!},
\end{align*}
and we get the assertion.
\end{proof}

\begin{ex}
\begin{itemize}
\item[{\rm (i)}] $B_{N,0}^{(r)}=B_{N,0}^{r+1}$
\item[{\rm (ii)}] $B_{N,1}^{(r)}=rB_{N,1}$
\item[{\rm (iii)}] $B_{N,2}^{(r)}=r B_{N,2}B_{N,0}^{N-1}+r(r-1)B_{N,1}^2B_{N,0}^{r-2}$
\end{itemize}
\label{ex2}
\end{ex}
%

%
%

\section{Applications by the Trudi's formula and inversion expressions} 

We can obtain different explicit expressions for the numbers $B_{N,n}^{(r)}$, $B_{N,n}$ and $B_{n}$ by using the Trudi's formula. We also show some inversion formulas.  The following relation is known as Trudi's formula \cite[Vol.3, p.214]{Muir},\cite{Trudi} and the case $a_0=1$ of this formula is known as Brioschi's formula \cite{Brioschi},\cite[Vol.3, pp.208--209]{Muir}.  

\begin{Lem}
For a positive integer $m$, we have 
\begin{multline*} 
\left|
\begin{array}{ccccc}
a_1  & a_2   &  \cdots   & \cdots& a_m  \\
a_{0}  & a_{1}    & \ddots   & &  \vdots \\
 &  \ddots&  \ddots  &  \ddots & \vdots  \\
  &     & \ddots  &a_1  & a_{2}  \\
&    &    & a_0  & a_1
\end{array}
\right|\\
=
\sum_{t_1 + 2t_2 + \cdots +mt_m=m}\binom{t_1+\cdots + t_m}{t_1, \dots,t_m}(-a_0)^{m-t_1-\cdots - t_m}a_1^{t_1}a_2^{t_2}\cdots a_m^{t_m}, \label{trudi}
\end{multline*}
where $\binom{t_1+\cdots + t_m}{t_1, \dots,t_m}=\frac{(t_1+\cdots + t_m)!}{t_1 !\cdots t_m !}$ are the multinomial coefficients. 
\label{lema0}
\end{Lem}

In addition, there exists the following inversion formula (see, e.g. \cite{KR}), which is based upon the relation: 
$$
\sum_{k=0}^n(-1)^{n-k}\alpha_k R(n-k)=0\quad(n\ge 1)\,. 
$$ 

\begin{Lem}
If $\{\alpha_n\}_{n\geq 0}$ is a sequence defined by $\alpha_0=1$ and 
$$
\alpha_n=\begin{vmatrix} R(1) & 1 & & \\
R(2) & \ddots &  \ddots & \\
\vdots & \ddots &  \ddots & 1\\
R(n) & \cdots &  R(2) & R(1) \\
 \end{vmatrix},  \ \text{then} \ R(n)=\begin{vmatrix} \alpha_1 & 1 & & \\
\alpha_2 & \ddots &  \ddots & \\
\vdots & \ddots &  \ddots & 1\\
\alpha_n & \cdots &  \alpha_2 & \alpha_1 \\
 \end{vmatrix}\,.
$$
Moreover, if 
$$
A=\begin{pmatrix} 
  1 &  & & \\
\alpha_1 & 1  &   & \\
\vdots & \ddots &  \ddots & \\
\alpha_n& \cdots &  \alpha_1 & 1 \\
 \end{pmatrix}, \ \text{then} \  A^{-1}=\begin{pmatrix} 
  1 &  & & \\
R(1) & 1  &   & \\
\vdots & \ddots &  \ddots & \\
R(n) & \cdots &  R(1) & 1 \\
 \end{pmatrix}\,.
$$
\label{lema}
\end{Lem}

From Trudi's formula, it is possible to give the combinatorial expression  
$$
\alpha_n=\sum_{t_1+2t_2+\cdots +n t_n=n}\binom{t_1+\cdots+t_n}{t_1, \dots, t_n}(-1)^{n-t_1-\cdots - t_n}R(1)^{t_1}R(2)^{t_2}\cdots R(n)^{t_n}\,.
$$
By applying these lemmata to Theorem \ref{th1-r}, we obtain an explicit expression for the generalized hypergeometric Bernoulli numbers $B_{N,n}^{(r)}$.   

\begin{theorem}\label{Turdi1}
For $n\geq 1$ 
\begin{multline*} 
B_{N,n}^{(r)}\\
=n!\sum_{t_1 + 2t_2 + \cdots + nt_n=n}\binom{t_1+\cdots + t_n}{t_1, \dots,t_n}(-1)^{t_1+\cdots+t_n}M_r(1)^{t_1}M_r(2)^{t_2}\cdots M_r(n)^{t_n}\,, 
\end{multline*}  
where $M_r(e)$ are given in $(\ref{mre})$. 
Moreover, 
$$
M_r(n)=\begin{vmatrix} -\frac{B_{N,1}^{(r)}}{1!} & 1 & & \\
\frac{B_{N,2}^{(r)}}{2!} & \ddots &  \ddots & \\
\vdots & \ddots &  \ddots & 1\\
\frac{(-1)^n B_{N,n}^{(r)}}{n!} & \cdots &  \frac{B_{N,2}^{(r)}}{2!} & -\frac{B_{N,1}^{(r)}}{1!} \\
 \end{vmatrix}\,,
$$
and  
{\small 
\begin{multline*} 
\begin{pmatrix} 1 &  &  & & \\
-\frac{B_{N,1}^{(r)}}{1!} & 1 &   &  & \\
\frac{B_{N,2}^{(r)}}{2!} & -\frac{B_{N,1}^{(r)}}{1!}  &  1  &  & \\
\vdots &  &  \ddots &  \ddots & \\
\frac{(-1)^n B_{N,n}^{(r)}}{n!} & \cdots &  \frac{B_{N,2}^{(r)}}{2!} & -\frac{B_{N,1}^{(r)}}{1!} &1 
\end{pmatrix}^{-1}\\ 
=\begin{pmatrix} 1 &  &  & & \\
M_r(1) & 1 &   &  & \\
M_r(2) & M_r(1)  &  1  &  & \\
\vdots &  &  \ddots &  \ddots & \\
M_r(n) & \cdots &  M_r(2) & M_r(1) &1 
\end{pmatrix}\,.
\end{multline*}  
}  
\label{th1234}
\end{theorem}

When $r=1$ in Theorem \ref{th1234}, we have  an explicit expression for the numbers $B_{N,n}$.  

\begin{Cor}  
For $n\geq 1$ 
\begin{multline*} 
B_{N,n} 
=n!\sum_{t_1 + 2t_2 + \cdots + nt_n=n}\binom{t_1+\cdots + t_n}{t_1, \dots,t_n}(-1)^{t_1+\cdots+t_n}\\
\times\left(\frac{N!}{(N+1)!}\right)^{t_1}\left(\frac{N!}{(N+2)!}\right)^{t_2}\cdots \left(\frac{N!}{(N+n)!}\right)^{t_n}  
\end{multline*}  
and  
$$
\frac{N!}{(N+n)!}=\begin{vmatrix} -\frac{B_{N,1}}{1!} & 1 & & \\
\frac{B_{N,2}}{2!} & \ddots &  \ddots & \\
\vdots & \ddots &  \ddots & 1\\
\frac{(-1)^n B_{N,n}}{n!} & \cdots &  \frac{B_{N,2}}{2!} & -\frac{B_{N,1}}{1!} \\
 \end{vmatrix}\,,
$$
\label{cor1}
\end{Cor}

When $r=N=1$ in Theorem \ref{th1234}, we have a different expression of the classical Bernoulli numbers. 

\begin{Cor}  
We have for $n\ge 1$ 
\begin{multline*} 
B_{n} 
=n!\sum_{t_1 + 2t_2 + \cdots + nt_n=n}\binom{t_1+\cdots + t_n}{t_1, \dots,t_n}(-1)^{t_1+\cdots+t_n}\\
\times\left(\frac{1}{2!}\right)^{t_1}\left(\frac{1}{3!}\right)^{t_2}\cdots \left(\frac{1}{(n+1)!}\right)^{t_n}  
\end{multline*}  
and  
$$
\frac{1}{(n+1)!}=\begin{vmatrix} -\frac{B_{1}}{1!} & 1 & & \\
\frac{B_{2}}{2!} & \ddots &  \ddots & \\
\vdots & \ddots &  \ddots & 1\\
\frac{(-1)^n B_{n}}{n!} & \cdots &  \frac{B_{2}}{2!} & -\frac{B_{1}}{1!} \\
 \end{vmatrix}.
$$
\label{cor2}
\end{Cor}

\section{Continued fractions of hypergeometric Bernoulli numbers}  

In \cite{AIK,Kaneko} by studying the convergents of the continued fraction of 
$$ 
\frac{x/2}{\tanh x/2}=\sum_{n=0}^\infty B_{2 n}\frac {x^{2 n}}{(2 n)!}\,,
$$
some identities of Bernoulli numbers are obtained.  
In this section, the $n$-th convergent of the generating function of hypergeometric Bernoulli numbers is explicitly given.  As an application, we give some identities of hypergeometric Bernoulli numbers in terms of binomial coefficients. 

The generating function on the left-hand side of (\ref{def:hgb}) can be expanded as a continued fraction  
\begin{equation}  
\frac{1}{{}_1 F_1(1;N+1;x)}
=1-\cfrac{x}{N+1+\cfrac{x}{N+2-\cfrac{(N+1)x}{N+3+\cfrac{2 x}{N+4-\cfrac{(N+2)x}{N+5+\ddots}}}}}   
\label{cf:hber}   
\end{equation} 
({\it Cf.} \cite[(91.2)]{Wall}).  
Its $n$-th convergent $P_n(x)/Q_n(x)$ ($n\ge 0$) is given by the recurrence relation 
\begin{align}
P_n(x)&=a_n(x)P_{n-1}(x)+b_n(x)P_{n-2}(x)~(n\ge 2),\\
Q_n(x)&=a_n(x)Q_{n-1}(x)+b_n(x)Q_{n-2}(x)~(n\ge 2),
\label{rec:hber}
\end{align} 
with initial values 
\begin{align*}
&P_0(x)=1,\quad P_1(x)=(N+1)-x;\\ 
&Q_0(x)=1,\quad Q_1(x)=N+1\,,  
\end{align*}
where for $n\ge 1$, $a_n(x)=N+n$, $b_{2n}(x)=n x$ and $b_{2 n+1}(x)=-(N+n)x$.

We have explicit expressions of both the numerator and the denominator of the $n$-th convergent of (\ref{cf:hber}).

\begin{theorem}  
For $n\ge 1$, we have  
\begin{align*}  
P_{2 n-1}(x)&=\sum_{j=0}^n(-1)^j\binom{n}{j}\prod_{l=1}^{2 n-j-1}(N+l)\cdot x^j\,,\\
P_{2 n}(x)&=\sum_{j=0}^n(-1)^j\binom{n}{j}\prod_{l=1}^{2 n-j}(N+l)\cdot x^j  
\end{align*}
and 
\begin{align*}  
Q_{2 n-1}(x)&=\sum_{j=0}^{n-1}\sum_{k=0}^j(-1)^{j-k}(2 n-j-1)_k\binom{n-k-1}{j-k}\prod_{l=k+1}^{2 n-j-1}(N+l)\cdot x^j\,,\\
Q_{2 n}(x)&=\sum_{j=0}^n\sum_{k=0}^j(-1)^{j-k}(2 n-j)_k\binom{n-k-1}{j-k}\prod_{l=k+1}^{2 n-j}(N+l)\cdot x^j\,.   
\end{align*}
\label{th1cf}
\end{theorem}  

\noindent 
{\it Remark.}  
Here we use the convenient values 
$$
\binom{n}{k}=0~(0\le n<k),\quad \binom{-1}{0}=1 
$$ 
and recognize the empty product as $1$. Otherwise, we should write $Q_{2 n}(x)$ as 
$$ 
Q_{2 n}(x)=\sum_{j=0}^{n-1}\sum_{k=0}^j(-1)^{j-k}(2 n-j)_k\binom{n-k-1}{j-k}\prod_{l=k+1}^{2 n-j}(N+l)\cdot x^j+n! x^n\,.
$$  
If we use the unsinged Stirling numbers of the first kind $\stf{n}{k}$, which generating function is given by 
$$
\sum_{n=k}^\infty(-1)^{n-k}\stf{n}{k}\frac{z^n}{n!}=\frac{\bigl(\log(1+z)\bigr)^k}{k!}\,,
$$ 
we can express the products as 
$$
\prod_{l=1}^{2 n-j-1}(N+l)=\sum_{i=1}^{2 n-j}\stf{2 n-j}{i}N^{i-1}
$$ 
or
$$ 
\prod_{l=k+1}^{2 n-j-1}(N+l)=\sum_{i=1}^{2 n-j-k}\stf{2 n-j-k}{i}(N+k)^{i-1}\,. 
$$ 

\begin{proof}[Proof of Theorem \ref{th1cf}]
The proof is done by induction on $n$. 
It is easy to see that for $n=0$ we have $P_0(x)=Q_0(x)=1$, and for $n=1$ we have $P_1(x)=(N+1)-x$ and $Q_1(x)=N+1$.  
Assume that the results hold up to $n-1(\ge 2)$.  Then by using the recurrence relation in (\ref{rec:hber}) 
\begin{align*} 
&(N+2 n)P_{2 n-1}(x)+n P_{2 n-2}(x)\cdot x\\
&=(N+2 n)\sum_{j=0}^n(-1)^j\binom{n}{j}\prod_{l=1}^{2 n-j-1}(N+l)\cdot x^j\\
&\quad +n\sum_{j=0}^{n-1}(-1)^j\binom{n-1}{j}\prod_{l=1}^{2 n-j-2}(N+l)\cdot x^{j+1}\\
&=(N+2 n)\prod_{l=1}^{2 n-1}(N+l)\\
&\quad +(N+2 n)\sum_{j=1}^n(-1)^j\binom{n}{j}\prod_{l=1}^{2 n-j-1}(N+l)\cdot x^j\\
&\quad -n\sum_{j=1}^{n}(-1)^j\binom{n-1}{j-1}\prod_{l=1}^{2 n-j-1}(N+l)\cdot x^j\,.
\end{align*}  
Since 
$$
(N+2 n)\binom{n}{j}-n\binom{n-1}{j-1}=(N+2 n-j)\binom{n}{j}\,,
$$ 
we get 
\begin{align*} 
&(N+2 n)P_{2 n-1}(x)+n P_{2 n-2}(x)\cdot x\\
&=\sum_{j=0}^n(-1)^j\binom{n}{j}\prod_{l=1}^{2 n-j}(N+l)\cdot x^j\\
&=P_{2 n}\,. 
\end{align*} 
Next, 
\begin{align*} 
&(N+2 n+1)P_{2 n}(x)-(N+n)P_{2 n-1}(x)\cdot x\\
&=(N+2 n)\sum_{j=0}^n(-1)^j\binom{n}{j}\prod_{l=1}^{2 n-j}(N+l)\cdot x^j\\
&\quad -(N+n)\sum_{j=0}^n(-1)^j\binom{n}{j}\prod_{l=1}^{2 n-j-1}(N+l)\cdot x^{j+1}\\
&=(N+2 n+1)\prod_{l=1}^{2 n}(N+l)\\
&\quad +(N+2 n+1)\sum_{j=1}^n(-1)^j\binom{n}{j}\prod_{l=1}^{2 n-j}(N+l)\cdot x^j\\
&\quad +(N+n)\sum_{j=1}^{n}(-1)^j\binom{n}{j-1}\prod_{l=1}^{2 n-j}(N+l)\cdot x^j\\
&\quad -(N+n)(-1)^n\prod_{l=1}^{n-1}(N+l)\cdot x^{n+1}\,.
\end{align*}  
Since 
$$
(N+2 n+1)\binom{n}{j}+(N+n)\binom{n}{j-1}=(N+2 n-j+1)\binom{n+1}{j}\,,
$$ 
we get 
\begin{align*} 
&(N+2 n+1)P_{2 n}(x)-(N+n)P_{2 n-1}(x)\cdot x\\
&=\sum_{j=0}^{n+1}(-1)^j\binom{n+1}{j}\prod_{l=1}^{2 n-j+1}(N+l)\cdot x^j\\
&=P_{2 n+1}\,. 
\end{align*} 
Concerning $Q_n(x)$, 
\begin{align*} 
&(N+2 n)Q_{2 n-1}(x)+n Q_{2 n-2}(x)\cdot x\\
&=(N+2 n)\sum_{j=0}^{n-1}\sum_{k=0}^j(-1)^{j-k}(2 n-j-1)_k\binom{n-k-1}{j-k}\prod_{l=k+1}^{2 n-j-1}(N+l)\cdot x^j\\
&\quad +n\sum_{j=0}^{n-1}\sum_{k=0}^j(-1)^{j-k}(2 n-j-2)_k\binom{n-k-2}{j-k}\prod_{l=k+1}^{2 n-j-2}(N+l)\cdot x^{j+1}\\
&=\prod_{l=1}^{2 n}(N+l)-n\sum_{k=0}^{n-1}(-1)^{n-k}(n)_k\binom{n-k-2}{n-k-1}\prod_{l=k+1}^{n-1}(N+l)\cdot x^{n}\\
&\quad +\sum_{j=1}^{n-1}\sum_{k=0}^j(-1)^{j-k}(2 n-j-1)_k\\
&\qquad \times\binom{n-k-1}{j-k}(N+2 n)\prod_{l=k+1}^{2 n-j-1}(N+l)\cdot x^j\\
&\quad -n\sum_{j=1}^{n-1}\sum_{k=0}^{j-1}(-1)^{j-k}n(2 n-j-1)_k\binom{n-k-2}{j-k-1}\prod_{l=k+1}^{2 n-j-1}(N+l)\cdot x^{j}\,. 
\end{align*}
Since $N+2 n=(N+2 n-j)+j$, 
$$
\prod_{l=k+1}^{2 n-j-1}(N+l)=\prod_{l=k+2}^{2 n-j}(N+l)-(2 n-j-k-1)\prod_{l=k+2}^{2 n-j-1}(N+l)\,, 
$$ 
$$
(j-k)\binom{n-k-1}{j-k}-n\binom{n-k-2}{j-k-1}=-(k+1)\binom{n-k-2}{j-k-1}
$$ 
and $(2 n-j-1)_k+k(2 n-j-1)_{k-1}=(2 n-j)_k$, we obtain that 
\begin{align*}
&(N+2 n)Q_{2 n-1}(x)+n Q_{2 n-2}(x)\cdot x\\
&=\prod_{l=1}^{2 n}(N+l)+n! x^n\\
&\quad +\sum_{j=1}^{n-1}\sum_{k=0}^j(-1)^{j-k}(2 n-j-1)_k\binom{n-k-1}{j-k}\prod_{l=k+1}^{2 n-j}(N+l)\cdot x^j\\
&\quad +\sum_{j=1}^{n-1}\sum_{k=0}^j(-1)^{j-k}j(2 n-j-1)_k\\
&\qquad \times\binom{n-k-1}{j-k}(N+2 n)\prod_{l=k+1}^{2 n-j-1}(N+l)\cdot x^j\\
&\quad -\sum_{j=1}^{n-1}\sum_{k=0}^{j-1}(-1)^{j-k}n(2 n-j-1)_k\\
&\qquad \times\binom{n-k-2}{j-k-1}(N+2 n)\prod_{l=k+1}^{2 n-j-1}(N+l)\cdot x^j\\
&=\prod_{l=1}^{2 n}(N+l)+n! x^n\\
&\quad +\sum_{j=1}^{n-1}\sum_{k=0}^j(-1)^{j-k}(2 n-j)_k\binom{n-k-1}{j-k}\prod_{l=k+1}^{2 n-j}(N+l)\cdot x^j\\
&=Q_{2 n}\,. 
\end{align*}  
Similarly, 
\begin{align*}  
&(N+2 n+1)Q_{2 n}(x)-(N+n)x Q_{2 n-1}(x)\\
&=(N+2 n+1)\sum_{j=0}^n\sum_{k=0}^j(-1)^{j-k}(2 n-j)_k\binom{n-k-1}{j-k}\prod_{l=k+1}^{2 n-j}(N+l)\cdot x^j\\
&\quad -(N+n)\sum_{j=0}^{n-1}\sum_{k=0}^j(-1)^{j-k}(2 n-j-1)_k\binom{n-k-1}{j-k}\prod_{l=k+1}^{2 n-j-1}(N+l)\cdot x^{j+1}\\
&=(N+2 n+1)\prod_{l=k+1}^{2 n}(N+l)\\
&\quad +\sum_{j=1}^n\sum_{k=0}^j(-1)^{j-k}(2 n-j)_k\binom{n-k-1}{j-k}\prod_{l=k+1}^{2 n-j+1}(N+l)\cdot x^j\\
&\quad \sum_{j=1}^n\sum_{k=0}^j(-1)^{j-k}(2 n-j)_k j\binom{n-k-1}{j-k}\prod_{l=k+1}^{2 n-j}(N+l)\cdot x^j\\ 
&\quad +\sum_{j=1}^{n}\sum_{k=0}^{j-1}(-1)^{j-k}(2 n-j)_k\binom{n-k-1}{j-k-1}\prod_{l=k+1}^{2 n-j+1}(N+l)\cdot x^{j}\\
&\quad -\sum_{j=1}^{n}\sum_{k=0}^{j-1}(-1)^{j-k}(2 n-j)_k(n-j+1)\binom{n-k-1}{j-k-1}\prod_{l=k+1}^{2 n-j}(N+l)\cdot x^{j}\,.
\end{align*}
Since 
$$
\binom{n-k-1}{j-k}+\binom{n-k-1}{j-k-1}=\binom{n-k}{j-k}
$$ 
and 
\begin{multline*}
j\binom{n-k-1}{j-k}-(n-j+1)\binom{n-k-1}{j-k-1}\\
=k\binom{n-k}{j-k}-(k+1)\binom{n-k-1}{j-k-1}\,, 
\end{multline*} 
we get 
\begin{align*} 
&(N+2 n+1)Q_{2 n}(x)-(N+n)x Q_{2 n-1}(x)\\
&=(N+2 n+1)\prod_{l=k+1}^{2 n}(N+l)\\
&\quad +\sum_{j=1}^n\sum_{k=0}^j(-1)^{j-k}(2 n-j)_k\binom{n-k}{j-k}\prod_{l=k+1}^{2 n-j+1}(N+l)\cdot x^j\\
&\quad +\sum_{j=1}^n\sum_{k=0}^j(-1)^{j-k}(2 n-j)_k\left(k\binom{n-k}{j-k}-(k+1)\binom{n-k-1}{j-k-1}\right)\\
&\qquad \times\prod_{l=k+1}^{2 n-j}(N+l)\cdot x^j\,.\\
\end{align*}
Since 
\begin{align*} 
&(-1)^{j-k-1}(2 n-j)_{k+1}\binom{n-k-1}{j-k-1}\prod_{l=k+2}^{2 n-j+1}(N+l)\\
&\quad +(-1)^{j-k}(2 n-j)_{k}\left(k\binom{n-k}{j-k}-(k+1)\binom{n-k-1}{j-k-1}\right)\prod_{l=k+1}^{2 n-j}(N+l)\\
&\quad +(-1)^{j-k+1}(2 n-j)_{k}\binom{n-k}{j-k}\prod_{l=k+1}^{2 n-j}(N+l)\\
&=(-1)^{j-k-1}(2 n-j+1)_{k+1}\binom{n-k-1}{j-k-1}\prod_{l=k+2}^{2 n-j+1}(N+l)\\
&\quad +(-1)^{j-k-1}(k+1)(2 n-j)_{k+1}\binom{n-k-1}{j-k-1}\prod_{l=k+2}^{2 n-j}(N+l)\,, 
\end{align*}
we have 
\begin{align*} 
&(N+2 n+1)Q_{2 n}(x)-(N+n)x Q_{2 n-1}(x)\\
&\quad +\sum_{j=1}^n\sum_{k=0}^j(-1)^{j-k}(2 n-j+1)_k\binom{n-k}{j-k}\prod_{l=k+1}^{2 n-j+1}(N+l)\cdot x^j\\
&=Q_{2 n+1}(x)\,. 
\end{align*}
\end{proof}

\subsection{Some more identities of hypergeometric Bernoulli numbers} 

Since $P_{2 n-1}(x)$, $P_{2 n}(x)$ and $Q_{2 n}(x)$ are the polynomials with degree $n$ and $Q_{2 n-1}(x)$ is the polynomial with degree $n-1$,  by the approximation property of the continued fraction, we have the following.   

\begin{Lem}  
Let $P_n(x)/Q_n(x)$ denote the $n$-th convergent of the continued fraction expansion of (\ref{cf:hber}).   
Then we have for $n\ge 0$ 
$$
Q_n(x)\sum_{\kappa=0}^\infty B_{N,\kappa}\frac{x^\kappa}{\kappa!}\equiv P_n(x)\pmod{x^{n+1}}\,.  
$$ 
\label{lem5}  
\end{Lem}  

By this approximation property,  the coefficients $x^j$ ($0\le j\le n$) of 
$$
Q_n(x)\sum_{\kappa=0}^\infty B_{N,\kappa}\frac{x^\kappa}{\kappa!}-P_n(x) 
$$ 
are nullified.   By Theorem \ref{th1cf}, 
\begin{align*}  
&Q_{2 n}(x)\sum_{\kappa=0}^\infty B_{N,\kappa}\frac{x^\kappa}{\kappa!}\\
&=\sum_{h=0}^\infty\sum_{j=0}^{\min\{h,n\}}\sum_{k=0}^j(-1)^{j-k}(2 n-j)_k\binom{n-k-1}{j-k}\prod_{l=k+1}^{2 n-j}(N+l)\cdot\frac{B_{N,h-j}}{(h-j)!}x^h 
\end{align*} 
and 
$$
P_{2 n}(x)=\sum_{h=0}^n(-1)^h\binom{n}{h}\prod_{l=1}^{2 n-h}(N+l)\cdot x^h\,. 
$$ 
Therefore,    
\begin{align*}
&\sum_{j=0}^{\min\{h,n\}}\sum_{k=0}^j(-1)^{j-k}(2 n-j)_k\binom{n-k-1}{j-k}\prod_{l=k+1}^{2 n-j}(N+l)\cdot\frac{B_{N,h-j}}{(h-j)!}\notag\\ 
&=\begin{cases}
(-1)^h\binom{n}{h}\prod_{l=1}^{2 n-h}(N+l)&(0\le h\le n);\\
0&(h>n)\,.   
\end{cases} 
\end{align*} 

Similarly, since 
\begin{align*}  
&Q_{2 n-1}(x)\sum_{\kappa=0}^\infty B_{N,\kappa}\frac{x^\kappa}{\kappa!}\\
&=\sum_{h=0}^\infty\sum_{j=0}^{\min\{h,n-1\}}\sum_{k=0}^j(-1)^{j-k}(2 n-j-1)_k\binom{n-k-1}{j-k}\\
&\qquad \times\prod_{l=k+1}^{2 n-j-1}(N+l)\cdot\frac{B_{N,h-j}}{(h-j)!}x^h 
\end{align*} 
and 
$$
P_{2 n-1}(x)=\sum_{h=0}^n(-1)^h\binom{n}{h}\prod_{l=1}^{2 n-h-1}(N+l)\cdot x^h\,, 
$$ 
we have    
\begin{align*}
&\sum_{j=0}^{\min\{h,n\}}\sum_{k=0}^j(-1)^{j-k}(2 n-j-1)_k\binom{n-k-1}{j-k}\prod_{l=k+1}^{2 n-j-1}(N+l)\cdot\frac{B_{N,h-j}}{(h-j)!}\notag\\ 
&=\begin{cases}
(-1)^h\binom{n}{h}\prod_{l=1}^{2 n-h-1}(N+l)&(0\le h\le n);\\
0&(h>n)\,.   
\end{cases} 
\end{align*}

\begin{theorem}  
We have 
\begin{align}
&\sum_{j=0}^{\min\{h,n\}}\sum_{k=0}^j(-1)^{j-k}(2 n-j)_k\binom{n-k-1}{j-k}\prod_{l=k+1}^{2 n-j}(N+l)\cdot\frac{B_{N,h-j}}{(h-j)!}\notag\\ 
&=\begin{cases}
(-1)^h\binom{n}{h}\prod_{l=1}^{2 n-h}(N+l)&(0\le h\le n);\\
0&(h>n)   
\end{cases} 
\label{eq:2n:0}
\end{align} 
and 
\begin{align}
&\sum_{j=0}^{\min\{h,n\}}\sum_{k=0}^j(-1)^{j-k}(2 n-j-1)_k\binom{n-k-1}{j-k}\prod_{l=k+1}^{2 n-j-1}(N+l)\cdot\frac{B_{N,h-j}}{(h-j)!}\notag\\ 
&=\begin{cases}
(-1)^h\binom{n}{h}\prod_{l=1}^{2 n-h-1}(N+l)&(0\le h\le n);\\
0&(h>n)\,.   
\end{cases} 
\label{eq:2n:1}
\end{align} 
\label{th10}
\end{theorem}  

In particular, when $N=1$, we have the relations for the classical Bernoulli numbers.   

\begin{Cor}  
We have 
\begin{align}
&\sum_{j=0}^{\min\{h,n\}}\sum_{k=0}^j(-1)^{j-k}(2 n-j)_k\binom{n-k-1}{j-k}\frac{(2 n-j+1)!}{(k+1)!(2 n-h+1)!}\cdot\frac{B_{h-j}}{(h-j)!}\notag\\ 
&=\begin{cases}
(-1)^h\binom{n}{h}&(0\le h\le n);\\
0&(h>n)   
\end{cases} 
\label{eq:2n:2}
\end{align} 
and 
\begin{align}
&\sum_{j=0}^{\min\{h,n\}}\sum_{k=0}^j(-1)^{j-k}(2 n-j-1)_k\binom{n-k-1}{j-k}\frac{(2 n-j)!}{(k+1)!(2 n-h)!}\cdot\frac{B_{h-j}}{(h-j)!}\notag\\ 
&=\begin{cases}
(-1)^h\binom{n}{h}&(0\le h\le n);\\
0&(h>n)\,.   
\end{cases} 
\label{eq:2n:3}
\end{align} 
\label{cor10} 
\end{Cor}

\noindent 
{\it Remark.}  
Since 
\begin{align*}  
&\sum_{k=0}^j(-1)^{j-k}\frac{(2 n-j)_k}{(k+1)!}\binom{n-k-1}{j-k}\\
&=\begin{cases} 
\displaystyle \frac{1}{j+1}\binom{n}{j}&\text{if $j$ is even};\\ 
\displaystyle \frac{1}{4 j}\binom{j-2}{(j-1)/2}^{-1}\binom{n-(j+1)/2}{(j-1)/2}\binom{n}{(j-1)/2}&\text{if $j$ is odd$\ge 3$};\\
\displaystyle \frac{1}{2}&\text{if $j=1$}\,, 
\end{cases}
\end{align*} 
we can write (\ref{eq:2n:2}) as 
\begin{align*}  
&\sum_{j=0}^{\fl{\frac{h}{2}}}\frac{(2 n-2 j+1)!}{2 j+1}\binom{n}{2 j}\frac{B_{h-2 j}}{(h-2 j)!}
+\frac{(2 n)!}{2}\frac{B_{h-1}}{(h-1)!}\\ 
&\qquad +\sum_{j=1}^{\fl{\frac{h-1}{2}}}\frac{(2 n-2 j)!}{4(2 j+1)}\binom{2 j-1}{j}^{-1}\binom{n-j-1}{j}\binom{n}{j}\frac{B_{h-2 j-1}}{(h-2 j-1)!}\\
&=\begin{cases}
(-1)^h\binom{n}{h}(2 n-h+1)!&\text{if $1\le h\le n$};\\
0&\text{if $n<h\le 2 n+1$}\,. 
\end{cases}  
\end{align*}    
Since 
\begin{align*}  
&\sum_{k=0}^j(-1)^{j-k}\frac{(2 n-j-1)_k}{(k+1)!}\binom{n-k-1}{j-k}\\
&=\begin{cases} 
\displaystyle \frac{\bigl((j/2)!\bigr)^2}{(j+1)!}\binom{n}{j/2}\binom{n-j/2-1}{j/2}&\text{if $j$ is even};\\ 
0&\text{if $j$ is odd}\,, 
\end{cases}
\end{align*} 
we can write (\ref{eq:2n:3}) as  
\begin{multline*} 
\sum_{j=0}^{\fl{\frac{h}{2}}}\frac{(j!)^2(2 n-2 j)!}{(2 j+1)!}\binom{n}{j}\binom{n-j-1}{j}\frac{B_{h-2 j}}{(h-2 j)!}\\ 
=\begin{cases}
(-1)^h\binom{n}{h}(2 n-h)!&\text{if $1\le h\le n$};\\
0&\text{if $n<h\le 2 n$}\,. 
\end{cases}  
\end{multline*}    
Here the empty summation is recognized as $0$, as usual.


\end{document}